\documentclass[reqno]{amsart}
\usepackage{amsmath}
\usepackage{arydshln}%
\allowdisplaybreaks[3]
\numberwithin{equation}{section}
\numberwithin{equation}{section}

\newtheorem{thm}{\indent Theorem}[section]

\newtheorem{lem}[thm]{\indent Lemma}
\newtheorem{prop}[thm]{\indent Proposition}

\newtheorem{rmk}{{\indent\bf Remark}}[section]
\newtheorem{expl}{{\indent\bf Example}}[section]

\newcommand{\mb}{\mbox}
\newcommand{\hs}{\hspace}

\newcommand{\vs}{\vspace}

\newcommand{\ttiny}{\fontsize{5pt}{\baselineskip}\selectfont}
\newcommand{\strl}[2]{\stackrel{\mbox{\ttiny $#1$}}{#2}}

\newcommand{\td}{\tilde}

\newcommand{\fr}{\frac}
\newcommand{\fkm}{{\mathfrak m}}
\newcommand{\fkp}{{\mathfrak p}}
\newcommand{\fkr}{{\mathfrak r}}

\newcommand{\edd}{\end{document}}
\newcommand{\be}{\begin{equation}}
\newcommand{\ee}{\end{equation}}
\newcommand{\bsl}{\backslash}

\newcommand{\lagl}{\langle}
\newcommand{\ragl}{\rangle}
\newcommand{\lmx}{\left(\begin{matrix}}
\newcommand{\rmx}{\end{matrix}\right)}
\newcommand{\ldt}{\left|\begin{matrix}}
\newcommand{\rdt}{\end{matrix}\right|}

\newcommand{\hess}{{\rm Hess\,}}

\newcommand{\tr}{{\rm tr\,}}

\newcommand{\const}{{\rm const}}

\newcommand{\bbr}{{\mathbb R}}

\newcommand{\bbh}{{\mathbb H}}

\newcommand{\bbs}{{\mathbb S}}
\newcommand{\mo}{M\"obius }
\newcommand{\ba}{\begin{array}}
\newcommand{\ea}{\end{array}}
\newcommand{\nnm}{\nonumber}
\newcommand{\beal}{\begin{align}}
\newcommand{\eal}{\end{align}}
\newcommand{\bea}{\begin{eqnarray}}
\newcommand{\eea}{\end{eqnarray}}

\newcommand{\spn}{{\rm Span\,}}

\newcommand{\stx}[2]{\strl{(#1)}{#2}}

\newcommand{\Ba}{\stx{1}{B}{}\!}
\newcommand{\Baa}{\stx{a}{B}{}\!}

\begin{document}

\title[Submanifolds in the unit sphere with
parallel Blaschke tensor]{On the immersed
submanifolds in the unit sphere\\ with parallel Blaschke tensor II} 

\author[X. X. Li]{Xingxiao Li} 

\author[H. R. Song]{Hongru Song ${}^*$} 

\dedicatory{}

%
\thanks{ 
Research supported by
National Natural Science Foundation of China (No. 11171091, 11371018).}
\address{
Department of Mathematics
\endgraf Henan Normal University \endgraf Xinxiang 453007, Henan
\endgraf P.R. China}
\email{xxl$@$henannu.edu.cn}

\address{
Department of Mathematics
\endgraf Henan Normal University \endgraf Xinxiang 453007, Henan
\endgraf P.R. China} %
\email{yaozheng-shr@163.com}


\thanks{${}^*$The corresponding author.}
\begin{abstract}
As is known, the Blaschke tensor $A$ (a symmetric covariant $2$-tensor) is one of the fundamental \mo invariants in the \mo differential geometry of submanifolds in the unit sphere $\bbs^n$, and the eigenvalues of $A$ are referred  to as the Blaschke eigenvalues. In this paper, we continue our job for the study on the submanifolds in $\bbs^n$ with parallel Blaschke tensors which we simply call {\em Blaschke parallel submanifolds} to find more examples and seek a complete classification finally. The main theorem of this paper is the classification of Blaschke parallel submanifolds in $\bbs^n$ with exactly three distinct Blaschke eigenvalues. Before
proving this classification we define, as usual, a new class of examples.
\end{abstract}

\maketitle

\section{Introduction}

Let $\bbs^n(r)$ be the standard $n$-dimensional sphere in the $(n+1)$-dimensional Euclidean space $\bbr^{n+1}$ of radius $r$,
and denote $\bbs^n=\bbs^n(1)$.
Let $\bbh^n(c)$ be the $n$-dimensional hyperbolic space of constant
curvature $c<0$ defined by
$$
\bbh^n(c)=\{y=(y_0,y_1)\in\bbr^{n+1}_1\,;\ \lagl y,y\ragl_1=\fr1c,\
y_0>0\},
$$
where, for any integer $N\geq 2$,
$\bbr^N_1\equiv\bbr_1\times\bbr^{N-1}$ is the $N$-dimensional
Lorentzian space with the standard Lorentzian inner product
$\lagl\cdot,\cdot\ragl_1$ given by
$$
\lagl y,y'\ragl_1=-y_0y'_0+y_1\cdot y'_1,\quad
y=(y_0,y_1),\,y'=(y'_0,y'_1)\in\bbr^N_1
$$
in which the dot ``$\cdot$'' denotes the standard Euclidean inner product
on $\bbr^{N-1}$. From now on, we simply write $\bbh^n$ for $\bbh^n(-1)$.

Denote by $\bbs^n_+$ the hemisphere in $\bbs^n$ whose first coordinate is
positive. Then there are two conformal diffeomorphisms
$$\sigma:\bbr^n\to \bbs^n\bsl\{(-1,0)\}\ \mb{ and }\ \tau:\bbh^n\to \bbs^n_+$$
defined as follows:

\begin{align}
\sigma(u)&=\left(\fr{1-|u|^2}{1+|u|^2},\fr{2u}{1+|u|^2}\right),\quad
u\in \bbr^n, \label{1.1}\\
\tau(y)&=\left(\fr1{y_0},\fr{y_1}{y_0}\right), \quad y=(y_0,y_1)\in
\bbh^n\subset\bbr^{n+1}_1. \label{1.2}
\end{align}

Let $x:M^m\to \bbs^{m+p}$ be an immersed umbilic-free submanifold in
$\bbs^{m+p}$. Without loss of generality, we usually assume that $x$ is linearly full, that is, $x$ can not be contained in a hyperplane in $\bbr^{m+p+1}$. Then it is known that there are four fundamental \mo invariants of $x$, in terms of the light-cone model established by C. P. Wang in 1998 (\cite{w}) that are the
\mo metric $g$, the Blaschke tensor $A$, the
\mo second fundamental form $B$ and the \mo form $C$. Since the pioneer work of Wang, there have been obtained many interesting results in the \mo geometry of submanifolds including some important classification theorems of submanifolds with particular \mo invariants, such as, the classification of surfaces
with vanishing \mo forms (\cite{lw2}), that of \mo
isotropic submanifolds (\cite{lwz}), that of hypersurfaces with constant \mo sectional curvature (\cite{gllmw}), that of \mo isoparametric hypersurfaces (\cite{llwz}, \cite{hld}, \cite{ltqw}, etc), and that of hypersurfaces with Blaschke tensors linearly dependent on
the \mo metrics and \mo second  fundamental forms \cite{lw1}, which is later generalized by \cite{lz05} and \cite{clq}, respectively, in two different directions. Here we should remark that, after the classification of all {\em \mo parallel hypersurfaces} in $\bbs^{m+1}$, that is, hypersurfaces with parallel \mo second fundamental forms (\cite{hl1}), Zhai-Hu-Wang recently proved in \cite{zhw} an interesting theorem which classifies all $2$-codimensional {\em \mo parallel submanifolds} in the unit sphere.

Clearly, it is much natural to study submanifolds in the unit sphere $\bbs^n$ with particular Blaschke tensors. Note that a submanifold in $\bbs^n$ with vanishing Blaschke tensor also has a
vanishing \mo form, and therefore is a special \mo isotropic
submanifold; any \mo isotropic submanifold is necessarily of parallel Blaschke tensor. Furthermore, all \mo parallel submanifolds also have vanishing \mo forms and parallel Blaschke tensors(\cite{zhw}). Thus a rather natural and interesting problem is to seek a classification of all the submanifolds with parallel Blaschke tensors which we shall call for simplicity {\em Blaschke parallel submanifolds}.

To this direction, the first step is indeed the study of hypersurfaces. In fact, the following theorem has been established:

\begin{thm}[\cite{lz06}]\label{hypcase}
Let $x:M^m\to \bbs^{m+1}$, $m\geq 2$, be a Blaschke parallel
hypersurface. Then the \mo form of $x$ vanishes identically and $x$ is either \mo parallel, or \mo isotropic, or \mo equivalent to one of the following examples which have exactly two distinct Blaschke eigenvalues:

$(1)$ one of the minimal hypersurfaces as indicated in
Example $3.2$ of \cite{lz06};

$(2)$ one of the non-minimal hypersurfaces as indicated in
Example $3.3$  of \cite{lz06}.
\end{thm}

As the second step, we have proved earlier the following classification:

\begin{thm}[\cite{ls3}]\label{ls3} Let $x:M^m\to \bbs^{m+p}$ be a Blaschke parallel submanifold immersed in $\bbs^{m+p}$ with vanishing \mo form $C$. If $x$ has two distinct Blaschke eigenvalues, then it
must be \mo equivalent to one of the following four kinds of immersions:

$(1)$ a non-minimal and umbilic-free pseudo-parallel immersion $\td x:M\to \bbs^{m+p}$ with parallel mean curvature and constant scalar curvature, which has two distinct principal curvatures in the direction of the mean curvature vector;

$(2)$ the image under $\sigma$ of a non-minimal and umbilic-free pseudo-parallel immersion $\bar x:M\to \bbr^{m+p}$ with parallel mean curvature and constant scalar curvature, which has two distinct principal curvatures in the direction of the mean curvature vector;

$(3)$ the image under $\tau$ of a non-minimal and umbilic-free pseudo-parallel immersion $\bar x:M\to \bbh^{m+p}$ with parallel mean curvature and constant scalar curvature, which has two distinct principal curvatures in the direction of the mean curvature vector;

$(4)$ a submanifold ${\rm LS}(m_1,p_1,r,{\mu})$ for some parameters $m_1,p_1,r,{\mu}$.
\end{thm}

{\rmk\rm Submanifolds ${\rm LS}(m_1,p_1,r,{\mu})$ with multiple parameters $m_1,p_1,r,{\mu}$ were first defined in Example 3.2 of \cite{ls3}. As in \cite{ls3}, we call a Riemannian submanifold {\em pseudo-parallel} if the inner product of its second fundamental form with the mean curvature vector is parallel. In particular, if the second fundamental form is itself parallel, then we simply call this submanifold {\em (Euclidean) parallel}.}

In this paper, we continue our work on the classification of the Blaschke parallel submanifolds in $\bbs^n$ with vanishing \mo forms. Naturally, due to Theorems \ref{hypcase} and \ref{ls3}, the next step is to study those Blaschke parallel submanifolds with three distinct Blaschke eigenvalues. To do this, we first construct in Section \ref{sec3} a new class of Blaschke parallel submanifolds denoted by ${\rm LS}(\fkm,\fkp,\fkr,{\mu})$ with, as desired,
vanishing \mo forms and exactly three distinct Blaschke eigenvalues. The idea of this construction originates from those hypersurface examples that were first introduced in \cite{lz06} (see also \cite{lz07}) and are the only non-\mo isoparametric but Blaschke isoparametric hypersurfaces (cf. \cite{lz09}) with two distinct Blaschke eigenvalues. Note that, due to \cite{ltw}, any Blaschke isoparametric hypersurfaces with more than two distinct Blaschke eigenvalues must be \mo isoparametric, which is an affirmative solution of the problem originally raised in \cite{lz09} (see also \cite{lp1} and \cite{lp2}). It should also be remarked that, by \cite{ltqw} and \cite{rt}, the \mo isoparametric hypersurfaces (cf. \cite{llwz}) have been completely classified and thus the work in \cite{ltw} actually finishes the classification of the Blaschke isoparametric hypersurfaces (see also the latest partial classification theorem in \cite{hlz}). Besides, there have been some parallel results on space-like hypersurfaces in the de Sitter space $\bbs^n_1$ (see \cite{ls1} and the references therein). Combining all we know on this subject, it turns out that our new examples and the argument in this present paper will shed a new light on the completement of our final classification work which will be done in a forth-coming paper.

The main theorem of this paper is now stated as follows:

\begin{thm}\label{main} Let $x:M^m\to \bbs^{m+p}$ be a Blaschke parallel submanifold immersed in $\bbs^{m+p}$ with vanishing \mo form $C$. If $x$ has three distinct Blaschke eigenvalues, then it
must be \mo equivalent to one of the following four kinds of immersions:

$(1)$ a non-minimal and umbilic-free pseudo-parallel immersion $\td x:M^m\to \bbs^{m+p}$ with parallel mean curvature and constant scalar curvature, which has three distinct principal curvatures in the direction of the mean curvature vector;

$(2)$ the image under $\sigma$ of a non-minimal and umbilic-free pseudo-parallel immersion $\bar x:M^m\to \bbr^{m+p}$ with parallel mean curvature and constant scalar curvature, which has three distinct principal curvatures in the direction of the mean curvature vector;

$(3)$ the image under $\tau$ of a non-minimal and umbilic-free pseudo-parallel immersion $\bar x:M^m\to \bbh^{m+p}$ with parallel mean curvature and constant scalar curvature, which has three distinct principal curvatures in the direction of the mean curvature vector;

$(4)$ a submanifold ${\rm LS}(\fkm,\fkp,\fkr,{\mu})$ given in Example \ref{expl3.2} for some multiple parameters $\fkm,\fkp,\fkr,{\mu}$ satisfying $m_3r^2_2\neq m_2r^2_3$.
\end{thm}

\begin{rmk}\label{rmk1}\rm In deed, it is directly verified that each of the immersed
submanifolds stated in Theorem \ref{main} is Blaschke parallel with vanishing \mo form and exactly three distinct Blaschke eigenvalues (see Section \ref{sec3}).

We also remark that the final classification theorem will be much like Theorem \ref{main} with the corresponding examples ${\rm LS}(\fkm,\fkp,\fkr,{\mu})$ being extended to the general case.
\end{rmk}

{\bf Acknowledgement}
The first author is supported by National Natural Science Foundation of China (No. 11171091, 11371018).

\section{Preliminaries}\label{sec2}

Let $x:M^m\to \bbs^{m+p}$ be an immersed umbilic-free submanifold. Denote by $h$ the second fundamental form of $x$
and $H=\fr1m\tr h$ the mean curvature vector field. Define
\be\label{2.1}
\rho=\left(\fr m{m-1}\left(|h|^2-m|H|^2\right)\right)^{\fr12},\quad
Y=\rho(1,x).
\ee
Then $Y:M^m\to \bbr^{m+p+2}_1$ is an immersion
of $M^m$ into the Lorentzian space $\bbr^{m+p+2}_1$ and is called the
canonical lift (or {\em the \mo position vector}) of $x$. The function
$\rho$ given by \eqref{2.1} may be called {\em the \mo factor} of the
immersion $x$. Denote
$$
C^{m+p+1}_+=\left\{y=(y_0,y_1)\in\bbr_1\times\bbr^{m+p+1}\,;\ \lagl
y,y\ragl_1=0,\ y_0>0\right\}
$$
and let $O(m+p+1,1)$ be the Lorentzian
group of all elements in $GL(m+p+2;\bbr)$ preserving the standard
Lorentzian inner product $\lagl\cdot,\cdot\ragl_1$ on
$\bbr^{m+p+2}_1$. There is a subgroup $O^+(m+p+1,1)$ of $O(m+p+1,1)$ that is given
by
\be\label{2.2}
O^+(m+p+1,1)=\left\{T\in O(m+p+1,1)\,;\ T(C^{m+p+1}_+)\subset
C^{m+p+1}_+\right\}.\ee

The following theorem is well known.

\begin{thm}\label{wth1} $($\cite{w}$)$ Two submanifolds $x,\td x:M^m\to \bbs^{m+p}$ with
\mo position vectors $Y,\td Y,$ respectively, are \mo equivalent if
and only if there is a $T\in O^+(m+p+1,1)$ such that $\td Y=T(Y)$.
\end{thm}

By Theorem \ref{wth1}, the induced metric
$g=Y^*\lagl\cdot,\cdot\ragl_1=\rho^2dx\cdot dx$ by $Y$ on
$M^m$ from the Lorentzian product $\lagl\cdot,\cdot\ragl_1$ is a \mo
invariant Riemannian metric (cf. \cite{b}, \cite{c}, \cite{w}), and
is called the \mo metric of $x$. Using the vector-valued function
$Y$ and the Laplacian $\Delta$ of the metric $g$, one can define
another important vector-valued function $N:M^m\to\bbr^{m+p+2}_1$, called {\em the \mo biposition vector}, by
\be\label{2.3}
N=-\fr1m\Delta Y-\fr1{2m^2}\lagl\Delta Y,\Delta Y\ragl_1Y.
\ee
Then it is verified that the \mo position vector $Y$ and the \mo
biposition vector $N$ satisfy the following identities \cite{w}:
\begin{align}
&\lagl \Delta Y,Y\ragl_1=-m,\quad \lagl \Delta Y,dY\ragl_1=0,\quad
\lagl\Delta Y,\Delta Y\ragl_1=1+m^2\kappa, \label{2.4}\\
&\lagl Y,Y\ragl_1=\lagl N,N\ragl_1=0,\quad \lagl Y,N\ragl_1=1,
\label{2.5}
\end{align}
where $\kappa$ denotes the normalized scalar curvature of the \mo
metric $g$.

Let $V\to M^m$ be the vector subbundle of the trivial Lorentzian
bundle $M^m\times\bbr^{m+p+2}_1$ defined to be the orthogonal
complement of $\bbr Y\oplus \bbr N\oplus Y_*(TM^m)$ with respect to
the Lorentzian product $\lagl\cdot,\cdot\ragl_1$. Then $V$ is called
the \mo normal bundle of the immersion $x$. Clearly, we have the
following vector bundle decomposition:
\be\label{2.6}
M^m\times\bbr^{m+p+2}_1=\bbr Y\oplus \bbr N\oplus Y_*(TM^m)\oplus
V.
\ee

Denote by $T^\bot M^m$ the normal bundle of the immersion
$x:M^m\to \bbs^{m+1}$. Then the mean curvature vector field $H$ of $x$
defines a bundle isomorphism $\Phi:T^\bot M^m\to V$ by
\be\label{2.7}
\Phi(e)=\left(H\cdot e,(H\cdot e)x+e\right)\quad\text{for any } e\in
T^\bot M^m.
\ee
It is known that $\Phi$ preserves the inner products as well as
the connections on $T^\bot M^m$ and $V$ (\cite{w}).

To simplify notations, we make the following conventions on the
ranges of indices used frequently in this paper:
\be\label{2.8}
1\leq i,j,k,\cdots\leq m,\quad m+1\leq
\alpha,\beta,\gamma,\cdots\leq m+p.
\ee

For a local orthonormal  frame field $\{ e_i\}$ for the induced
metric $dx\cdot dx$ with the dual $\{\theta^i\}$ and for an orthonormal
normal frame field $\{ e_\alpha\}$ of $x$, we set
\be\label{2.9}
E_i=\rho^{-1}e_i,\quad \omega^i=\rho\theta^i,\quad
E_\alpha=\Phi(e_\alpha).
\ee
Then $\{E_i\}$ is a local orthonormal frame field on $M^m$ with respect to the \mo metric
$g$, $\{\omega^i\}$ is the dual of $\{E_i\}$, and $\{E_\alpha\}$ is
a local orthonormal frame field of the \mo normal bundle $V\to M$. Clearly,
$\{Y,N,Y_i:=Y_*(E_i),E_\alpha\}$ is a moving frame of $\bbr^{m+p+2}_1$ along $M^m$. If the basic \mo invariants $A$, $B$ and $C$ are respectively written as
\be\label{2.10}
A=\sum
A_{ij}\omega^i\omega^j,\quad  B=\sum B^\alpha_{ij}\omega^i\omega^j
E_\alpha,\quad C=\sum  C^\alpha_i\omega^i E_\alpha,
\ee
then we have the following equations of motion (\cite{w}):
\begin{align}
dY=&\sum Y_i\omega^i,\quad dN=\sum A_{ij}\omega^jY_i+C^\alpha_i\omega^iE_\alpha,\label{2-6}\\
dY_i=&-\sum A_{ij}\omega^jY-\omega^iN+\sum \omega^j_iY_j+\sum B^\alpha_{ij}\omega^jE_\alpha,\label{2-7}\\
dE_\alpha=&-\sum C^\alpha_i\omega^iY-\sum B^\alpha_{ij}\omega^jY_i+\sum \omega^\beta_\alpha E_\beta,\label{2-8}
\end{align}
where $\omega^j_i$ are the Levi-Civita connection forms of the \mo metric $g$ and $\omega^\beta_\alpha$ are the (\mo) normal connection forms of $x$. Furthermore,
by a direct computation one can find the following local expressions (\cite{w}):
\begin{align}
A_{ij}=&- \rho^{-2}\left(\hess_{ij}(\log \rho)- e_i(\log \rho)
e_j(\log \rho)
-\sum  H^\alpha h^\alpha_{ij}\right) \nnm\\
&-\fr12\rho^{-2}\left(|d\log \rho|^2-1+| H|^2\right)\delta_{ij},\label{2.12}\\
B^\alpha_{ij}=&\ \rho^{-1}\left( h^\alpha_{ij}-
H^\alpha\delta_{ij}\right), \label{2.13}\\
C^\alpha_i=&-\rho^{-2}\left(H^\alpha_{,i}+\sum(
h^\alpha_{ij}-H^\alpha\delta_{ij})
e_j(\log\rho)\right), \label{2.11}
\end{align}
in which the subscript ``$,i$'' denotes the covariant derivative
with respect to the induced metric $d x\cdot d x$ and in the
direction $e_i$.

Denote, respectively, by $R_{ijkl}$, $R^\bot_{\alpha\beta ij}$ the components
of the \mo Riemannian curvature tensor and the
curvature operator of the \mo normal bundle
with respect to the tangent frame field $\{E_i\}$ and the \mo normal frame field $\{E_\alpha\}$. Then we have (\cite{w})
\begin{align} \tr
A&=\fr1{2m}(1+m^2\kappa),\quad\tr B=\sum
B^\alpha_{ii}E_\alpha=0,\quad |B|^2=\sum
(B^\alpha_{ij})^2=\fr{m-1}{m}. \label{2.14}
\\
&R_{ijkl}=\sum(B^\alpha_{il}B^\alpha_{jk}
-B^\alpha_{ik}B^\alpha_{jl})+A_{il}\delta_{jk}-A_{ik}\delta_{jl}
+A_{jk}\delta_{il}-A_{jl}\delta_{ik}. \label{2.15}\\
&\hs{2.6cm}R^\bot_{\alpha\beta ij}=\sum(B^\alpha_{jk}B^\beta_{ik}-B^\alpha_{ik}B^\beta_{jk}
).\label{2-16}
\end{align}

We should remark that both equations \eqref{2.15} and \eqref{2-16} have the opposite sign from those in \cite{w} due to the different notations
of the Riemannian curvature tensor. Furthermore, let
$A_{ijk}$, $B^\alpha_{ijk}$ and $C^\alpha_{ij}$ denote, respectively,
the components with respect to the frame fields $\{E_i\}$ and
$\{E_\alpha\}$ of the covariant derivatives of $A$, $B$ and $C$, then the following Ricci identities hold (\cite{w}):
\begin{align}
A_{ijk}-A_{ikj}=&\sum
(B^\alpha_{ik}C^\alpha_j-B^\alpha_{ij}C^\alpha_k), \label{2.17}\\
B^\alpha_{ijk}-B^\alpha_{ikj}=&\delta_{ij}C^\alpha_k
-\delta_{ik}C^\alpha_j, \label{2.18}\\
C^\alpha_{ij}-C^\alpha_{ji}
=&\sum(B^\alpha_{ik}A_{kj}-B^\alpha_{kj}A_{ki}).  \label{2.16}
\end{align}

Denote by $R_{ij}$ the components of the Ricci curvature. Then
by taking trace in \eqref{2.15} and \eqref{2.18}, one obtains
\begin{align}
&R_{ij}=-\sum B^\alpha_{ik}B^\alpha_{kj}+\delta_{ij}\tr
A+(m-2)A_{ij}, \label{2.19}\\
&(m-1)C^\alpha_i=-\sum B^\alpha_{ijj}. \label{2.20}
\end{align}

Moreover, for the higher order covariant derivatives $B^\alpha_{ij\cdots kl}$, we have the following Ricci identities:
\be\label{2-21}
B^\alpha_{ij\cdots kl}-B^\alpha_{ij\cdots lk}=\sum B^\alpha_{qj\cdots}R_{iqkl} +\sum B^\alpha_{iq\cdots}R_{jqkl}+\cdots-\sum B^\beta_{ij\cdots}R^\bot_{\beta\alpha kl}.
\ee

By \eqref{2.14}, \eqref{2.19} and \eqref{2.20}, if $m\geq
3$, then the Blaschke tensor $A$ and the \mo form $C$ are
determined by the \mo metric $g$, \mo second fundamental form
$B$ and the (\mo) normal connection of $x$. Thus the following theorem holds:

\begin{thm}[cf. \cite{w}]\label{wth} Two submanifolds $x:M^m\to
\bbs^{m+p}$ and $\td x:\td M^m\to \bbs^{m+p}$, $m\geq 3$, are \mo
equivalent if and only if they have the same \mo metrics, the same \mo second fundamental forms and the same (\mo) normal connections.
\end{thm}

\section{The new examples}\label{sec3}

Before proving the main theorem, we need to find more examples of Blaschke parallel submanifolds in the unit sphere $\bbs^{m+p}$ as many as possible with parallel Blaschke tensors and with three distinct Blaschke eigenvalues. We note that, by Zhai-Hu-Wang (\cite{zhw}), all \mo parallel submanifolds in $\bbs^{m+p}$ are necessarily Blaschke parallel ones. This kind of examples are listed in \cite{zhw}. In this section we define a new class of Blaschke parallel examples which are in general not \mo parallel.

{\expl\label{expl3.1}\rm The following three classes of submanifolds have been studied in \cite{ls3} (cf. \cite{zhw}).

$(1)$ The umbilic-free pseudo-parallel submanifolds $\td x:M^m\to \bbs^{m+p}$ with parallel mean curvature $\td H$ and constant scalar curvature $\td S$.

$(2)$ The composition $\td x=\sigma\circ\bar x$ where $\bar x:M^m\to\bbr^{m+p}$ is an umbilic-free pseudo-parallel submanifolds with parallel mean curvature $\bar H$ and constant scalar curvature $\bar S$.

$(3)$ The composition $\td x:=\tau\circ\bar x$ where $\bar x:M^m\to\bbh^{m+p}$ is an umbilic-free pseudo-parallel submanifold with parallel mean curvature $\bar H$ and constant scalar curvature $\bar S$.

\begin{rmk}\label{rmk3.1}\rm It is shown in \cite{ls3} that all the examples $\td x:M^m\to \bbs^{m+p}$ given in (1), (2) and (3) above are Blaschke parallel with vanishing \mo form. Furthermore, $\td x$ in (1) has three distinct Blaschke eigenvalues if and only if it is not minimal and has three distinct principal curvatures in the direction of the mean curvature vector $\td H$, while $\td x$ in (2) (resp. in (3)) has three distinct Blaschke eigenvalues if and only if the corresponding $\bar x:M^m\to\bbr^{m+p}$ (resp. $\bar x:M^m\to\bbh^{m+p}$) is not minimal and has three distinct principal curvatures in the direction of the mean curvature vector $\bar H$. Note that $\td x$ is \mo isotropic, or equivalently, $\td x$ has only one distinct Blaschke eigenvalue, if and only $\td x$ in (1), or $\bar x$ in (2) or in (3) is minimal (\cite{lwz}). In addition,
it is not hard to see that (\cite{zhw}) a submanifold $\td x$ is \mo parallel if and only if $\td x$ in (1), or $\bar x$ in (2) or in (3) is (Euclidean) parallel.\end{rmk}

{\expl\label{expl3.2}\rm Submanifolds ${\rm LS}(\fkm,\fkp,\fkr,{\mu})$.

We start with a multiple parameter data $(\fkm,\fkp,\fkr,{\mu})$ where
$$\fkm:=(m_1,m_2,m_3),\quad\fkp:=(p_1,p_2,p_3),\quad\fkr:=(r_1,r_2,r_3),\quad\mu:=({\mu}_1,{\mu}_2,{\mu}_3)$$
with $m_1,m_2,m_3$ and $p_1,p_2,p_3$ being integers satisfying
$$
m_1,m_2,m_3\geq 1,\quad p_1,p_2,p_3\geq 0;
$$
and with $r_1,r_2,r_3$ and ${\mu}_1,{\mu}_2,{\mu}_3$ being real numbers satisfying
$$
r_1,r_2,r_3>0,\quad r_1^2=r_2^2+r_3^2,\quad {\mu}_1,{\mu}_2,{\mu}_3\geq 0,\ \quad\mu_1+{\mu}_2+{\mu}_3=1.
$$
Denote
$$m:=m_1+m_2+m_3,\quad p:=p_1+p_2+p_3+1.$$}

Since
$$
\det\ldt
m+m_1&m_2&m_3\\ m_1& m+m_2 &m_3\\ m_1&m_2& m+m_3
\rdt=2m^3\neq 0,
$$
there exist real numbers $\lambda_1,\lambda_2,\lambda_3$ that are uniquely determined by
\be\label{lambda}\begin{cases}(m+m_1)\lambda_1+m_2\lambda_2+m_3\lambda_3 =-\fr{m_1}{r^2_1},&\\
m_1\lambda_1+(m+m_2)\lambda_2+m_3\lambda_3=\fr{m_2}{r^2_2},&\\
m_1\lambda_1+m_2\lambda_2+(m+m_3)\lambda_3=\fr{m_3}{r^2_3}.&
\end{cases}\ee
Let $B^0_1,B^0_2,B^0_3$ be real numbers defined by
\be\label{3.2}
m_1B^0_1+m_2B^0_2+m_3B^0_3=0,\quad
B^0_aB^0_b=-(\lambda_a+\lambda_b),\quad a\neq b
\ee
where $1\leq a,b\leq 3$. Clearly, $B^0_1,B^0_2,B^0_3$ are unique up to a common sign. In fact we have
\be\label{b01}
(B^0_a)^2=\fr1{m_a}((m_{a'}+m_{a''})\lambda_a+m_{a'}\lambda_{a'}+m_{a''}\lambda_{a''})
\ee
where $a,a',a''$ is an even permutation of $1,2,3$.

Note that, by \eqref{3.2}
\be\label{3.4}
2\lambda_a=(\lambda_a+\lambda_{a'})+(\lambda_a+\lambda_{a''})-(\lambda_{a'}+\lambda_{a''}) =-B^0_aB^0_{a'}-B^0_aB^0_{a''}+B^0_{a'}B^0_{a''}\ee
for a permutation $a,a',a''$ of $1,2,3$.

Then the following lemma can be shown by a direct computation using \eqref{lambda}, \eqref{3.2}, \eqref{b01} and \eqref{3.4}:
\begin{lem}\label{lem3.1} It holds that
\begin{align}
&2\lambda_1+(B^0_1)^2=-\fr1{r^2_1},\quad 2\lambda_2+(B^0_2)^2=\fr1{r^2_2},\quad 2\lambda_3+(B^0_3)^2=\fr1{r^2_3},\label{3.5}\\
&-\fr{m_1-1}{r^2_1}+(B^0_1)^2=(m+m_1-2)\lambda_1+m_2\lambda_2+m_3\lambda_3
=(m-2)\lambda_1+\sum_am_a\lambda_a,\label{3.6}\\
&\fr{m_2-1}{r^2_2}+(B^0_2)^2=m_1\lambda_1+(m+m_2-2)\lambda_2+m_3\lambda_3
=(m-2)\lambda_2
+\sum_am_a\lambda_a,\label{3.7}\\
&\fr{m_3-1}{r^2_3}+(B^0_3)^2=m_1\lambda_1+m_2\lambda_2+(m+m_3-2)\lambda_3
=(m-2)\lambda_3+\sum_am_a\lambda_a,\label{3.8}\\
&-\fr{m_1(m_1-1)}{r^2_1}+\fr{m_2(m_2-1)}{r^2_2}+\fr{m_3(m_3-1)}{r^2_3}\nnm\\ &\hs{1cm}=(2m_1(m-1)-(m+m_1))\lambda_1+(2m_2(m-1)-(m+m_2))\lambda_2\nnm\\
&\hs{1cm}\quad +(2m_3(m-1)-(m+m_3))\lambda_3,\label{3.9}\\
&m_1(B^0_1)^2+m_2(B^0_2)^2+m_3(B^0_3)^2 =(m+m_1)\lambda_1+(m+m_2)\lambda_2+(m+m_3)\lambda_3,\label{3.10}\\
&-r^2_1B^0_1+r^2_2B^0_2+r^2_3B^0_3=0,\label{3.11}\\
&-r^2_1(B^0_1)^2+r^2_2(B^0_2)^2+r^2_3(B^0_3)^2 =-\lambda_1r^2_1+\lambda_2r^2_2+\lambda_3r^2_3=1.\label{3.12}
\end{align}
\end{lem}

Let
$$\td y=(\td y_0,\td y_1):M_1\to
\bbh^{m_1+p_1}\left(-\fr1{r^2_1}\right)\subset\bbr^{m_1+p_1+1}_1$$ be an immersed
minimal submanifold of dimension $m_1$ with constant scalar
curvature
\be\label{3.20}\td S_1=-\fr{m_1(m_1-1)}{r^2_1}+{\mu}_1\left(m_1(B^0_1)^2 +m_2(B^0_2)^2+m_3(B^0_3)^2-\fr{m-1}{m}\right),\ee
and
\be\label{3.21}
\td y_2: M_2\to \bbs^{m_2+p_2}(r_2)\subset\bbr^{m_2+p_2+1},\quad
\td y_3: M_3\to \bbs^{m_3+p_3}(r_3)\subset\bbr^{m_3+p_3+1}
\ee
be two immersed minimal submanifolds of dimensions $m_2$, $m_3$
with constant scalar curvatures
\begin{align}&\td S_2=\fr{m_2(m_2-1)}{r^2_2}+{\mu}_2\left(m_1(B^0_1)^2 +m_2(B^0_2)^2 +m_3(B^0_3)^2-\fr{m-1}{m}\right),\label{3.1}\\
&\td S_3=\fr{m_3(m_3-1)}{r^2_3}+{\mu}_3\left(m_1(B^0_1)^2 +m_2(B^0_2)^2 +m_3(B^0_3)^2-\fr{m-1}{m}\right),\label{3.1-1}
\end{align}
respectively. Then by \eqref{3.9} and \eqref{3.10}
\begin{align}
\td S_1+\td S_2+\td S_3=&-\fr{m_1(m_1-1)}{r_1^2}+\fr{m_2(m_2-1)}{r_2^2}+\fr{m_3(m_3-1)}{r_3^2}\nnm\\
&-\fr{m-1}{m}+m_1(B^0_1)^2 +m_2(B^0_2)^2+m_3(B^0_3)^2\nnm\\
=&2(m-1)\sum_am_a\lambda_a-\fr{m-1}{m}.
\label{S1+S2}
\end{align}

Set
\be\label{3.22}
\td M^m=M_1\times M_2\times M_3,\quad\td Y=(\td y_0,\td y_1,\td
y_2,\td y_3).
\ee
Then $\td Y:\td M^m\to\bbr^{m+p+2}_1$ is an immersion satisfying $\lagl
\td Y,\td Y\ragl_1=0$ with the induced Riemannian metric
$$g=\lagl d\td Y,d\td Y\ragl_1=-d\td y_0^2+d\td y^2_1+d\td y^2_2+d\td y^2_3.$$
Thus
\be\label{3.23}
(\td M^m,g)=\left(M_1,\lagl d\td y,d\td y\ragl_1\right)\times
\left(M_2,d\td y^2_2\right)\times
\left(M_3,d\td y^2_3\right)\ee
as Riemannian manifolds. Define
\be\label{3.24}
\td x_1=\fr{\td y_1}{\td y_0},\quad \td
x_2=\fr{\td y_2}{\td y_0},\quad \td
x_3=\fr{\td y_3}{\td y_0}, \quad \td x=(\td x_1,\td x_2,\td x_3).
\ee
Then $\td x^2=1$ and thus $\td x:\td M^m\to \bbs^{m+p}$ is an immersed submanifold which we denote simply by ${\rm LS}(\fkm,\fkp,\fkr,{\mu})$. Since
\be\label{3.25}
d\td x=-\fr{d\td y_0}{\td y^2_0}(\td y_1,\td y_2,\td y_3) +\fr1{\td
y_0}(d\td y_1,d\td y_2,d\td y_3), \ee
the induced metric $\td g=d\td x\cdot d\td x$ on $\td M^m$ is related to $g$ by
\be\label{3.26}
\td g=\td y^{-2}_0(-d\td y^2_0+d\td y^2_1+d\td y^2_2+d\td y^2_3)=\td
y^{-2}_0g.
\ee

Denote
\be\label{e0}
\bar E_{\alpha0}:=-(B^0_1\td y_0,B^0_1\td y_1,B^0_2\td y_2,B^0_3\td y_3),
\ee
and let
\begin{align*}
&\{\bar E_\alpha; m+1\leq\alpha\leq m+p_1\},\\
&\{\bar E_\alpha; m+p_1+1\leq\alpha\leq m+p_1+p_2\},\\
&\{\bar E_\alpha; m+p_1+p_2+1\leq\alpha\leq m+p_1+p_2+p_3\}
\end{align*}
be orthonormal normal frame fields of $\td y$, $\td y_2$,  $\td y_3$, respectively, with
$$\bar E_\alpha=(\bar E_{\alpha 0},\bar E_{\alpha 1})\in\bbr^1_1\times\bbr^{m_1+p_1}\equiv\bbr^{m_1+p_1+1}_1, \mb{ for }\alpha=m+1,\cdots,m+p_1.$$
Define
\begin{align}
\td e_\alpha=&(\bar E_{\alpha 1},0,0)-\bar E_{\alpha 0} \td x\in\bbr^{m_1+p_1}\times\bbr^{m_2+p_2+1}\times\bbr^{m_3+p_3+1}\equiv\bbr^{m+p+1},\nnm\\
&\mb{for }\alpha=m+1,\cdots,m+p_1;\label{3-10}\\
\td e_\alpha=&(0,\bar E_\alpha,0)\in\bbr^{m_1+p_1}\times\bbr^{m_2+p_2+1} \times\bbr^{m_3+p_3+1}\equiv\bbr^{m+p+1},\nnm\\
&\mb{for }\alpha=m+p_1+1,\cdots,m+p_1+p_2;\label{3-11}\\
\td e_\alpha=&(0,0,\bar E_\alpha)\in\bbr^{m_1+p_1}\times\bbr^{m_2+p_2+1} \times\bbr^{m_3+p_3+1}\equiv\bbr^{m+p+1},\nnm\\
&\mb{for }\alpha=m+p_1+p_2+1,\cdots,m+p_1+p_2+p_3;\label{3-11.1}\\
\td e_{\alpha_0}=&-(B^0_1\td y_1,B^0_2\td y_2,B^0_3\td y_3)+B^0_1\td y_0\td x.\label{3-11.2}
\end{align}
Then, by Lemma \ref{lem3.1}, $\{\td e_\alpha,\td e_{\alpha0};\ m+1\leq\alpha\leq m+p-1\}$ is an orthonormal normal frame field of ${\rm LS}(\fkm,\fkp,\fkr,{\mu})$.

Hence, by \eqref{3.25}, for $\alpha=m+1,\cdots,m+p_1$
\begin{align} d \td e_\alpha\cdot d\td x=&\ (d\bar E_{\alpha1},0,0)\cdot d\td x-d\bar E_{\alpha 0}\td x d\td x-\bar E_{\alpha 0}d\td x^2\nnm\\
=&\ \td y^{-1}_0(-d\bar E_{\alpha 0}d\td y_0+d\bar E_{\alpha 1}\cdot d\td y_1)-\bar E_{\alpha 0}\td
y^{-2}_0g,\label{3.27}
\end{align}
where the third equality comes from the fact that
\be\label{3.28}
-\bar E_{\alpha 0}\td y_0+\bar E_{\alpha 1}\cdot \td
y_1=-d\bar E_{\alpha 0}\td y_0+d\bar E_{\alpha 1}\cdot\td y_1=0;
\ee
while
\be\label{3.27-1}
d \td e_\alpha\cdot d\td x=\td y^{-1}_0(d\bar E_\alpha\cdot d\td y_2)
\ee
for $\alpha=m+p_1+1,\cdots,m+p_1+p_2$, and
\be\label{3.27-2}
d \td e_\alpha\cdot d\td x=\td y^{-1}_0(d\bar E_\alpha\cdot d\td y_3)
\ee
for $\alpha=m+p_1+p_2+1,\cdots,m+p-1$. Furthermore,
\begin{align} d \td e_{\alpha 0}\cdot d\td x=&\td y^{-2}_0dy_0(B^0_1d\td y_1,B^0_2d\td y_2,B^0_3d\td y_3)\cdot (\td y_1,\td y_2,\td y_3)\nnm\\
&-\td y^{-1}_0(B^0_1d\td y_1,B^0_2d\td y_2,B^0_3d\td y_3)\cdot (d\td y_1,d\td y_2,d\td y_3)\nnm\\
&+B^0_1d\td y_0\td x d\td x+B^0_1\td y_0d\td x d\td x\nnm\\
=&-\td y^{-1}_0((B^0_2-B^0_1)d\td y^2_2+(B^0_3-B^0_1)d\td y^2_3).\label{3.27-3}
\end{align}
It then follows that, if we denote by
$$
\bar h_{M_1}=\sum_{\alpha=m+1}^{m+p_1}\bar h^\alpha \bar E_\alpha,\quad \bar h_{M_2}=\sum_{\alpha=m+p_1+1}^{m+p_1+p_2}\bar h^\alpha \bar E_\alpha,\quad
\bar h_{M_3}=\sum_{\alpha=m+p_1+p_2+1}^{m+p-1}\bar h^\alpha \bar E_\alpha$$
the second fundamental forms of $\td y$, $\td y_2$ and $\td y_3$, respectively, then the second fundamental form
$$\td h=\sum_{\alpha=m+1}^{m+p-1}\td h^\alpha \td e_\alpha+\td h^{\alpha_0}\td e_{\alpha_0}$$
of ${\rm LS}(\fkm,\fkp,\fkr,{\mu})$ is given as follows:
\begin{align}
&\td h^\alpha=-d \td e_\alpha\cdot d\td x=y^{-1}_0\bar h^\alpha+\bar E_{\alpha 0}\td y^{-2}_0g, \quad \mb{for }\alpha=m+1,\cdots,m+p_1;\label{3.29}\\
&\td h^\alpha=-d \td e_\alpha\cdot d\td x=\td
y^{-1}_0\bar h^\alpha,\quad \mb{for }\alpha=m+p_1+1,\cdots,m+p-1;\label{3.29-1}\\
&\td h^{\alpha_0}=\td y^{-1}_0((B^0_2-B^0_1)d\td y^2_2+(B^0_3-B^0_1)d\td y^2_3).
\label{3.29-2}
\end{align}

Let
$$\{E_i\,;1\leq i\leq m_1\},\quad \{E_i\,;m_1+1\leq i\leq
m_1+m_2\},\quad \{E_i\,;m_1+m_2+1\leq i\leq
m\}$$
be local orthonormal frame fields for
$$(M_1,\lagl d\td y,d\td y\ragl_1), \quad (M_2,d\td y^2_2),\quad (M_3,d\td y^2_3),$$
respectively. Then $\{E_i\,;1\leq i\leq m\}$ is a local
orthonormal frame field for $(M^m,g)$.

Put $\td e_i=\td y_0E_i$,
$i=1,\cdots,m$. Then $\{\td e_i\,;1\leq i\leq m\}$ is a local
orthonormal frame field for $(M^m,\td g)$. Thus for $\alpha=m+1,\cdots,m+p_1$,
\be\label{3.30}
\left\{\aligned
\td h^\alpha_{ij}=&\ \td h^\alpha(\td e_i,\td e_j)=\td y^2_0\td h^\alpha(E_i,E_j)=\td y_0
\bar h^\alpha(E_i,E_j)+\bar E_{\alpha 0}\,g(E_i,E_j)\\
=&\ \td y_0\bar h^\alpha_{ij}+\bar E_{\alpha 0}\delta_{ij},
\quad \mb{when\ } 1\leq i,j\leq m_1,\\
\td h^\alpha_{ij}=&\ \bar E_{\alpha 0}\delta_{ij},\quad  \mb{otherwise};
\endaligned\right.
\ee
while
\be\label{3.37}
\left\{\aligned
\td h^\alpha_{ij}=&\ \td h^\alpha(\td e_i,\td e_j)=\td y^2_0\td h^\alpha(E_i,E_j)=\td y_0
\bar h^\alpha(E_i,E_j)=\td y_0\bar h^\alpha_{ij},\\
&\mb{when\ } m_1+1\leq i,j\leq m_1+m_2,\\
\td h^\alpha_{ij}=&\ 0,\quad  \mb{otherwise}
\endaligned\right.
\ee
for $\alpha=m+p_1+1,\cdots,m+p_1+p_2$, and
\be\label{3.9-1}
\left\{\aligned
\td h^\alpha_{ij}=&\ \td h^\alpha(\td e_i,\td e_j)=\td y^2_0\td h^\alpha(E_i,E_j)=\td y_0
\bar h^\alpha(E_i,E_j)=\td y_0\bar h^\alpha_{ij},\\
&\mb{when\ } m_1+m_2+1\leq i,j\leq m,\\
\td h^\alpha_{ij}=&\ 0,\quad \mb{otherwise}
\endaligned\right.
\ee
for $\alpha=m+p_1+p_2+1,\cdots,m+p-1$. Furthermore
\be\label{3.9-2}
\td h^{\alpha_0}_{ij}=\td h^{\alpha_0}(\td e_i,\td e_j)=\td y^2_0\td h^{\alpha_0}(E_i,E_j)=\begin{cases}
\td y_0
(B^0_2-B^0_1)\delta_{ij},
&\mb{for\ } m_1+1\leq i,j\leq m_1+m_2,\\
\td y_0 (B^0_3-B^0_1)\delta_{ij},
&\mb{for\ } m_1+m_2+1\leq i,j\leq m,\\
0,& \mb{otherwise}.
\end{cases}
\ee

Since $\td y$, $\td y_2$ and $\td y_3$ are minimal, the mean
curvature
$$\td H=\fr1m\left(\sum_{\alpha=m+1}^{m+p-1}\sum_{i=1}^m\td h^\alpha_{ii}\td e_\alpha+\sum_{i=1}^m\td h^{\alpha_0}_{ii}\td e_{\alpha_0}\right)$$
of ${\rm LS}(\fkm,\fkp,\fkr,{\mu})$ is given by
\begin{align}
\td H^\alpha=&\fr1m\sum_{i=1}^m\td h^\alpha_{ii}
=\fr{\td y_0}m\sum_{i=1}^{m_1}\bar h^\alpha_{ii}+\bar E_{\alpha 0}
=\bar E_{\alpha 0},\quad\mb{ for }m+1\leq\alpha\leq m+p_1;\label{3.31}\\
\td H^\alpha=&\fr1m\sum_{i=1}^m\td h^\alpha_{ii}
=\fr{\td y_0}m\sum_{i=m_1+1}^{m_1+m_2}\bar h^\alpha_{ii}=0,\quad\mb{ for }m+p_1+1\leq\alpha\leq m+p_1+p_2;\label{3.31-1}\\
\td H^\alpha=&\fr1m\sum_{i=1}^m\td h^\alpha_{ii}
=\fr{\td y_0}m\sum_{i=m_1+m_2+1}^{m}\bar h^\alpha_{ii}=0,\quad\mb{ for }m+p_1+p_2+1\leq\alpha\leq m+p-1;\label{3.31-2}\\
\td H^{\alpha_0}=&\fr1m\sum_{i=1}^m\td h^{\alpha_0}_{ii}
=\fr{\td y_0}m(m_2(B^0_2-B^0_1)+m_3(B^0_3-B^0_1))=-\td y_0 B^0_1.\label{3.31-3}
\end{align}

From \eqref{3.2}, \eqref{3.9}, \eqref{3.10}, \eqref{S1+S2}, \eqref{3.30}--\eqref{3.31-3} and the Gauss equations of $\td y$, $\td y_2$ and $\td y_3$, we find
\begin{align}
|\td h|^2=&\td y^2_0\sum_{\alpha=m+1}^{m+p_1}\sum_{i,j=1}^{m_1}(\bar h^\alpha_{ij})^2+m\sum_{\alpha=m+1}^{m+p_1}(\bar E_{\alpha 0})^2+\td y^2_0\sum_{\alpha=m+p_1+1}^{m+p_1+p_2}\sum_{i,j=m_1+1}^{m_1+m_2}(\bar h^\alpha_{ij})^2\nnm\\
&+\td y^2_0\sum_{\alpha=m+p_1+p_2+1}^{m+p-1}\sum_{i,j=m_1+m_2+1}^m(\bar h^\alpha_{ij})^2+\td y^2_0(m_2(B^0_2-B^0_1)^2+m_3(B^0_3-B^0_1)^2)\nnm\\
=&\fr{m-1}{m}\td y^2_0+m\sum_{\alpha=m+1}^{m+p_1}(\bar E_{\alpha 0})^2 +m\td y^2_0(B^0_1)^2,\\
|\td H|^2=&\sum_{\alpha=m+1}^{m+p_1}(\td H^\alpha)^2 +\sum_{\alpha=m+p_1+1}^{m+p_1+p_2}(\td H^\alpha)^2
+\sum_{\alpha=m+p_1+p_2+1}^{m+p-1}(\td H^\alpha)^2 +(\td H^{\alpha_0})^2\nnm\\
=&\sum_{\alpha=m+1}^{m+p_1}(\bar E_{\alpha 0})^2+\td y^2_0(B^0_1)^2.
\end{align}
It then follows that
$$|\td h|^2-m|\td
H|^2=\fr{m-1}{m}\td y^2_0>0,$$
implying that $\td x$ is umbilic-free, and the \mo factor $\td \rho=\td y_0$. So $\td Y$ is the \mo position of ${\rm LS}(\fkm,\fkp,\fkr,{\mu})$. Consequently,
the \mo metric of ${\rm LS}(\fkm,\fkp,\fkr,{\mu})$ is nothing but $\lagl d\td Y,d\td Y
\ragl_1=g$. Furthermore, if we denote by $\{\omega^i\}$ the local coframe field on $M^m$ dual to $\{E_i\}$, then the \mo second fundamental form
$$
\td B=\sum_{\alpha=m+1}^{m+p}\td B^\alpha \Phi(\td e_\alpha)\equiv
\sum_{\alpha=m+1}^{m+p}\td B^\alpha_{ij}\omega^i\omega^j \Phi(\td e_\alpha)
$$ of ${\rm LS}(\fkm,\fkp,\fkr,{\mu})$ is given by
\begin{align}
\td B^\alpha=&\td \rho^{-1}\sum(\td h^\alpha_{ij}-\td H^\alpha\delta_{ij})\omega^i\omega^j
=\sum_{i,j=1}^{m_1}\bar h^\alpha_{ij}\omega^i\omega^j,\nnm\\
&\ \mb{ for }\alpha=m+1,\cdots,m+p_1; \label{3.32}\\
\td B^\alpha=&\td \rho^{-1}\sum(\td h^\alpha_{ij}-\td H^\alpha\delta_{ij})\omega^i\omega^j
=\sum_{i,j=m_1+1}^{m_1+m_2}\bar h^\alpha_{ij}\omega^i\omega^j,\nnm\\
&\ \mb{ for }\alpha=m+p_1+1,\cdots,m+p_1+p_2, \label{3.32_1}\\
\td B^\alpha=&\td \rho^{-1}\sum(\td h^\alpha_{ij}-\td H^\alpha\delta_{ij})\omega^i\omega^j
=\sum_{i,j=m_1+m_2+1}^{m}\bar h^\alpha_{ij}\omega^i\omega^j,\nnm\\
&\ \mb{ for }\alpha=m+p_1+p_2+1,\cdots,m+p-1, \label{3.32_2}\\
\td B^{\alpha_0}=&B^0_1\sum_{i=1}^{m_1}(\omega^i)^2 +B^0_2\sum_{i=m_1+1}^{m_1+m_2}(\omega^i)^2
+B^0_3\sum_{i=m_1+m_2+1}^{m}(\omega^i)^2, \label{3.32_3}
\end{align}
or, equivalently
\be\label{B}
\td B^\alpha_{ij}=\begin{cases} \bar h^\alpha_{ij},& \mb{if } m+1\leq\alpha\leq m+p_1,\ 1\leq i,j\leq m_1,\\
&\mb{or } m+p_1+1\leq\alpha\leq m+p_1+p_2,\ m_1+1\leq i,j\leq m_1+m_2,\\
&\mb{or } m+p_1+p_2+1\leq\alpha\leq m+p-1,\ m_1+m_2+1\leq i,j\leq m,\\
B^0_1\delta_{ij},&\mb{if } \alpha=\alpha_0,\ 1\leq i,j\leq m_1,\\
B^0_2\delta_{ij},&\mb{if } \alpha=\alpha_0,\ m_1+1\leq i,j\leq m_1+m_2,\\
B^0_3\delta_{ij},&\mb{if } \alpha=\alpha_0,\ m_1+m_2+1\leq i,j\leq m\\
0,&\mb{otherwise.}
\end{cases}
\ee

On the other hand, since the \mo metric $g$ is the direct product of $\lagl d\td y,d\td y\ragl_1$, $d\td y_2\cdot d\td y_2$ and $d\td y_3\cdot d\td y_3$, one finds by the minimality and the Gauss equations of $\td y$, $\td y_2$
and $\td y_3$ that the Ricci tensor of $g$ is given as
follows:
\begin{align}
R_{ij}=&-\fr{m_1-1}{r^2_1}\delta_{ij}
-\sum_{\alpha_1}\sum_{k=1}^{m_1}\bar h^{\alpha_1}_{ik}\bar h^{\alpha_1}_{kj},\quad\mb{if\
} 1\leq
i,j\leq m_1, \label{3.33}\\
R_{ij}=&\ \fr{m_2-1}{r^2_2}\delta_{ij} -\sum_{\alpha_2}\sum_{k=m_1+1}^{m_1+m_2}\bar h^{\alpha_2}_{ik}\bar h^{\alpha_2}_{kj},\quad\mb{if\ } m_1+1\leq
i,j\leq m_1+m_2, \label{3.34}\\
R_{ij}=&\ \fr{m_3-1}{r^2_3}\delta_{ij} -\sum_{\alpha_3}\sum_{k=m_1+m_2+1}^{m}\bar h^{\alpha_3}_{ik}\bar h^{\alpha_3}_{kj},\quad\mb{if\ } m_1+m_2+1\leq
i,j\leq m, \label{3.34-1}\\
R_{ij}=&\ 0,\quad\mb{otherwise}, \label{3.35}
\end{align}
where
\begin{align*}&m+1\leq\alpha_1\leq m+p_1,\quad m+p_1+1\leq\alpha_2\leq m+p_1+p_2,\\ &m+p_1+p_2+1\leq\alpha_3\leq m+p-1.
\end{align*}
On the other hand,
by the definitions of $\td y$,  $\td y_2$ and $\td y_3$, the trace of $A$ is given by
\be\label{3.56}
\tr A=\fr1{2m}(1+m^2\kappa)=\sum_am_a\lambda_a.
\ee
Since $m\geq 3$, it follows by \eqref{2.19}, \eqref{3.6} and \eqref{B}--\eqref{3.56} that the Blaschke tensor of ${\rm LS}(\fkm,\fkp,\fkr,{\mu})$ is given by $A=\sum
A_{ij}\omega^i\omega^j$ where, for $1\leq i,j\leq m_1$,
\begin{align}
A_{ij}
=&\fr1{m-2}\left(-\fr{m_1-1}{r^2} +(B^0_1)^2-\sum_am_a\lambda_a\right)\delta_{ij}
=\lambda_1\delta_{ij}.\label{3.38}
\end{align}
Similarly,
\begin{align}
A_{ij}=&\lambda_2\delta_{ij},\quad\mb{for } m_1+1\leq
i,j\leq m_1+m_2, \label{3.39}\\
A_{ij}=&\lambda_3\delta_{ij},\quad\mb{for } m_1+m_2+1\leq
i,j\leq m, \label{3.39-1}\\
A_{ij}=&\ 0,\quad\mb{otherwise}. \label{3.40}
\end{align}

Therefore, $A$ has constant eigenvalues $\lambda_1,\lambda_2,\lambda_3$.
It follows that ${\rm LS}(\fkm,\fkp,\fkr,{\mu})$ is Blaschke parallel since $\omega^j_i=0$ for $A_{ii}\neq A_{jj}$.

\begin{prop} For each of the submanifolds ${\rm LS}(\fkm,\fkp,\fkr,{\mu})$ defined in Example \ref{expl3.2}, we have

$(1)$ The \mo form $C$ vanishes identically;

$(2)$ The Blaschke eigenvalues $\lambda_1,\lambda_2,\lambda_3$ are distinct if and only if $m_3r^2_2\neq m_2r^2_3$;

$(3)$ The \mo second fundamental form $B$ is parallel if and only if
$$\td y:M_1\to \bbh^{m_1+p_1}\left(-\fr1{r^2_1}\right),\  \td y_2:M_2\to \bbs^{m_2+p_2}(r_2)\ \mb{ and }\ \td y_3:M_3\to \bbs^{m_3+p_3}(r_3)$$
are all parallel as Riemannian submanifolds. Furthermore, if it is the case, then $\td y(M_1)$ is isometric to the totally geodesic hyperbolic space $\bbh^{m_1}\left(-\fr1{r^2_1}\right)$ and $\td y$ can be taken as the standard embedding of $\bbh^{m_1}\left(-\fr1{r^2_1}\right)$ in $\bbh^{m_1+p_1}\left(-\fr1{r^2_1}\right)$.
\end{prop}

\begin{proof} The proof of (1) and (3) is omitted here since it is similar to that of Proposition 3.1 in \cite{ls3}; The conclusion (2) is direct from \eqref{lambda}.\end{proof}

\begin{rmk}\rm It is not hard to show that, if $m_3r^2_2=m_2r^2_3$, then $\lambda_2=\lambda_3$. In this case, ${\rm LS}(\fkm,\fkp,\fkr,{\mu})$ has two distinct Blaschke eigenvalues and is included as one special case of Example 3.1 or Example 3.2 in \cite{ls3}.\end{rmk}

\section{Proof of the main theorem}

Let $x:M^m\to \bbs^{m+p}$ be an umbilic-free submanifold in $\bbs^{m+p}$ satisfying all the conditions in the main theorem, and $\lambda_1,\lambda_2,\lambda_3$ be the three distinct Blaschke eigenvalues of $x$. Since the \mo form $C\equiv 0$ and the Blaschke tensor $A$ is parallel, $(M,g)$ is isometric to a direct product of three Riemannian manifolds $(M_1,g^{(1)})$, $(M_2,g^{(2)})$ and $(M_3,g^{(3)})$ with
$$m_1:=\dim M_1,\quad m_2:=\dim M_2, \quad m_3:=\dim M_3$$
such that, under the orthonormal frame field $\{E_i\}$ of $(M^m,g)$ satisfying
$$E_1,\cdots,E_{m_1}\in TM_1,\ E_{m_1+1},\cdots,E_{m_1+m_2}\in TM_2,\ E_{m_1+m_2+1},\cdots,E_m\in TM_3,$$
the components $A_{ij}$ of $A$ with respect to $\{E_i\}$ are diagonalized as follows:
\be\label{4-1}
A_{i_1j_1}=\lambda_1\delta_{i_1j_1},\ A_{i_2j_2}=\lambda_2\delta_{i_2j_2},\ A_{i_3j_3}=\lambda_3\delta_{i_3j_3},\ A_{i_1j_2}=A_{i_2j_3}=A_{i_1j_3}=0,
\ee
where and from now on we agree with
$$
1\leq i_1,j_1,k_1,\cdots\leq m_1,\  m_1+1\leq i_2,j_2,k_2,\cdots\leq m_1+m_2,\ m_1+m_2+1\leq i_3,j_3,k_3,\cdots\leq m.
$$

Furthermore, as done in Section \ref{sec2}, write $B=\sum B_{ij}^\alpha\omega^i\omega^j E_\alpha$ for some \mo normal frame field $\{E_\alpha\}$, where $\{\omega^i\}$ is the dual of $\{E_i\}$. Then, by $C\equiv 0$ and \eqref{2.16}, the corresponding components $B^\alpha_{ij}$ satisfy
\be\label{4-2}
B^\alpha_{i_1i_2}=B^\alpha_{i_1i_3}=B^\alpha_{i_2i_3}\equiv 0,\mb{ for all }\alpha,i_1,i_2,i_3.
\ee

In general, we have
\begin{lem}\label{lem4.1} It holds that
\be\label{4-2.1}
B^\alpha_{ij\cdots k}\equiv 0,
\ee
if there exist two of the indices $i,j,\cdots,k$ assuming the forms $i_a$, $i_b$ with $a\neq b$,
where $ij\cdots k$ denotes a multiple index of order no less than $2$.
\end{lem}

\begin{proof}
Due to \eqref{4-2} and the method of induction, it suffices to prove that if \eqref{4-2.1} holds then
\be\label{4-2.2}
B^\alpha_{ij\cdots kl}\equiv 0\ee
for indices $i,j,\cdots,k,l$ in which there exist two assuming the forms $i_a, i_b$ with $a\neq b$.

In fact, we only need to consider the following two cases:

(i) There exist two of the indices $i,j,\cdots,k$ which assume the forms $i_a$, $i_b$ with $a\neq b$.

In this case, we use \eqref{4-2.1} and $\omega^{j_b}_{i_a}=0$ $(a\neq b)$ to find
$$
B^\alpha_{ij\cdots kl}\omega^l=dB^\alpha_{ij\cdots k}-\sum B^\alpha_{lj\cdots k}\omega^l_i
-\sum B^\alpha_{il\cdots k}\omega^l_j-\cdots -\sum B^\alpha_{ij\cdots l}\omega^l_k
+\sum B^\beta_{ij\cdots k}\omega_\beta^\alpha\\
\equiv 0.
$$
So \eqref{4-2.2} is true.

(ii) Either $1\leq i,j,\cdots,k\leq m_1$, or $m_1+1\leq i,j,\cdots,k\leq m_1+m_2$, or $m_1+m_2+1\leq i,j,\cdots,k\leq m$.

Without loss of generality, we assume the first. Then it must be that $l=j_a$ for $a=2$ or $a=3$.
Note that by \eqref{2-16} and \eqref{4-2},
\be\label{4-2.3}
R^\bot_{\alpha\beta i_1j_a}=\sum_q(B^\alpha_{j_aq}B^\beta_{i_1q}-B^\alpha_{i_1q}B^\beta_{j_aq} )\equiv 0,\quad\forall i_1,j_a,\ a=2,3. \ee
This together with Case (i), the Ricci identities \eqref{2-21} and the fact that $R_{i_1j_aij}\equiv 0$ shows that
$$
B^\alpha_{ij\cdots kj_a}=B^\alpha_{ij\cdots j_ak}+\sum B^\alpha_{qj\cdots}R_{iqkj_a} +\sum B^\alpha_{iq\cdots}R_{jqkj_a}+\cdots-\sum B^\beta_{ij\cdots}R^\bot_{\beta\alpha kj_a}
\equiv 0.
$$
\end{proof}

\begin{lem}\label{lem4.2}It holds that, for all $i_a,j_a,k_a,\cdots,l_a$, $i_b,j_b,\cdots,k_b$ and $1\leq a\neq b\leq 3$,
\begin{align}
&\sum_\alpha B^\alpha_{i_aj_a}B^\alpha_{i_bj_b} =-(\lambda_a+\lambda_b)\delta_{i_aj_a}\delta_{i_bj_b},\label{4-3}\\
&\sum_\alpha B^\alpha_{i_aj_ak_a}B^\alpha_{i_bj_b} =0.\label{4-4}
\end{align}
More generally,
\be\label{4-5}
B^\alpha_{i_aj_ak_a\cdots l_a}B^\alpha_{i_bj_b\cdots k_b} =0,
\ee
where $i_aj_ak_a\cdots l_a$ is a multiple index of order no less than $3$.
\end{lem}

\begin{proof} This lemma mainly comes from the \mo Gauss equation \eqref{2.15} and the parallel assumption of the Blaschke tensor $A$. In fact, since $a\neq b$, \eqref{4-3} is given by \eqref{2.15}, \eqref{4-1}, \eqref{4-2} and that $R_{i_ai_bj_b j_a}\equiv 0$; \eqref{4-4} is given by \eqref{2.15}, \eqref{4-2.1}, $R_{i_ai_bj_b j_a}\equiv 0$ and the parallel of $A$; Finally, \eqref{4-5} can be shown by the method of induction using \eqref{4-4} and Lemma \ref{lem4.1}.
\end{proof}

As the corollary of \eqref{4-3}, we have for $a\neq b$
\begin{align}
&\sum_\alpha B^\alpha_{i_aj_a}(B^\alpha_{i_bi_b}-B^\alpha_{j_bj_b})=0,\label{4-6}\\
&\sum_\alpha B^\alpha_{i_aj_a}B^\alpha_{i_bj_b}=0,\mb{ if }i_a\neq j_a.\label{4-7}
\end{align}

Define
\begin{align}
&V_a=\spn\left\{\sum_\alpha B^\alpha_{i_aj_a\cdots k_a}E_\alpha\right\},\quad a=1,2,3;\\
&V_{a0}=V_a\cap(V_{a'}+V_{a''})^\bot,\quad\mb{ so that }\quad V_{a0}\,\bot\, V_{b0}\ \mb{ for } a\neq b,
\label{4-8}
\end{align}
where, as mentioned earlier, $a,a',a''$ is an even permutation of $1,2,3$.

Let $V'_{a0}$ ($a=1,2,3$) be the orthogonal complement of $V_{a0}$ in $V_a$ and denote
$$V_0:=V'_{10}+V'_{20}+V'_{30}.$$

\begin{lem}\label{lem4.3} It holds that $1\leq\dim V_0\leq 2$.\end{lem}

\begin{proof}
For any $i,j$, we denote by $B^{V'_{a0}}_{ij}$ the $V'_{a0}$-component of $B_{ij}$, $a=1,2,3$. Then it follows from \eqref{4-6} and \eqref{4-7} that
\be\label{4.12}
B^{V'_{a0}}_{i_aj_a}=0,\quad
B^{V'_{a0}}_{i_ai_a}=B^{V'_{a0}}_{j_aj_a},\ \mb{ for any }\ i_a,j_a,\ i_a\neq j_a.
\ee
So that
\be\label{4.13}
V'_{a0}=\spn\{B^{V'_{a0}}_{i_ai_a}\},\ \mb{for each fixed}\ i_a,\  a=1,2,3.
\ee
In particular, $\dim V'_{a0}\leq 1$, $a=1,2,3$.

On the other hand, by the second equation in \eqref{2.14}, we have
\be\label{4.14}
\sum_{i_1}B_{i_1i_1}+\sum_{i_2}B_{i_2i_2}+\sum_{i_3}B_{i_3i_3}=0.
\ee
But, for any $a$ and $i_a$,
$$
B_{i_ai_a}=B^{V_{a0}}_{i_ai_a}+B^{V'_{a0}}_{i_ai_a}=B^{V_{a0}}_{i_ai_a}+B^{V_0}_{i_ai_a},
$$
and $V_{a0}\,\bot V_{b0}$ for $a\neq b$, so that \eqref{4.14} reduces to
\be\label{4.15}
\sum_{i_a}B^{V_{a0}}_{i_ai_a}=0,\quad a=1,2,3,\quad
\sum_{i_1}B^{V_0}_{i_1i_1}+\sum_{i_2}B^{V_0}_{i_2i_2}+\sum_{i_3}B^{V_0}_{i_3i_3}=0.
\ee
The second equality in \eqref{4.15} together with \eqref{4.12} shows that, for fixed $i_1$, $i_2$ and $i_3$,
\be\label{4.17} m_1B^{V_0}_{i_1i_1}+m_2B^{V_0}_{i_2i_2}+m_3B^{V_0}_{i_3i_3}=0
\ee
which with \eqref{4.13} proves that $\dim V_0\leq 2$.

Finally, if $\dim V_0=0$, then for any $a\neq b$, $B_{i_ai_a}\bot B_{i_bi_b}$ which with \eqref{4-3} imples that $\lambda_a+\lambda_b=0$, contradicting the assumption that $\lambda_1,\lambda_2,\lambda_3$ are distinct.
\end{proof}

Clearly by definition,  $V_0$, $V_{10}$, $V_{20}$ and $V_{30}$ are orthogonal to each other. Denote $\iota:=\dim V_0$. Then we can properly choose an orthonormal normal frame field $\{E_\alpha\}$ such that
\be\label{4-12}
E_{\alpha^\nu}\in V_0,\quad \nu=1,\cdots,\iota;\quad E_{\alpha_a},E_{\beta_a},E_{\gamma_a},\cdots \in V_{a0},\ \mb{ for } a=1,2,3.
\ee

{\lem\label{lem4.4} $V_0$, $V_{10}$, $V_{20}$ and $V_{30}$ are parallel in the \mo normal bundles $V$. In particular, they are all of constant dimension.}

\begin{proof} Let $\xi_a$ and $\xi_0$ be sections of $V$ such that $\xi_a\in V_{a0}$ ($a=1,2,3$) and $\xi_0\in V_0$. Then, by the definition of the subspaces $V_0$, $V_{10}$, $V_{20}$ and $V_{30}$, $\xi_a$ (resp. $\xi_0$) is a linear combination of $B^{V_{a0}}_{i_aj_a\cdots k_a}$ (resp. of $B^{V_0}_{i_1i_1}$, $B^{V_0}_{i_2i_2}$ and $B^{V_0}_{i_3i_3}$ for some $i_1,i_2,i_3$). Thus, by \eqref{4-5}--\eqref{4-7} and \eqref{4.15}, it is not hard to conclude that
\be\label{4.18}\lagl D^\bot\xi_a,\xi_0\ragl_1=\lagl D^\bot\xi_a,\xi_b\ragl_1\equiv 0,\quad b\neq a.\ee
In fact, we take, for example, $a=1$ and $\xi_a=m_1B^{V_{10}}_{11}$. Write
$B_{i_1i_1}=\sum_\alpha B^\alpha_{i_1i_1}E_\alpha$.
Since by \eqref{4.15},
\begin{align*}
m_1B^{V_{10}}_{11}=&(B^{V_{10}}_{11}-B^{V_{10}}_{22}) +\cdots+(B^{V_{10}}_{11}-B^{V_{10}}_{m_1m_1})
=\sum_{\alpha}\left((B^\alpha_{11}-B^\alpha_{22})+\cdots+(B^\alpha_{11}-B^\alpha_{m_1m_1})\right)E_\alpha,
\end{align*}
we have
\begin{align*}
D^\bot\xi_a=&m_1D^\bot B^{V_{10}}_{11}=\sum_{\alpha}\left((dB^\alpha_{11}-dB^\alpha_{22})+\cdots +(dB^\alpha_{11}-dB^\alpha_{m_1m_1})\right)E_\alpha\\ &+\sum_{\alpha,\beta}\left((B^\alpha_{11}-B^\alpha_{22})+\cdots +(B^\alpha_{11}-B^\alpha_{m_1m_1})\right)\omega^\alpha_\beta E_\alpha\\
=&\sum_{i,\alpha}\left((B^\alpha_{11i}-B^\alpha_{22i})+\cdots +(B^\alpha_{11i}-B^\alpha_{m_1m_1i})\right)\omega^iE_\alpha\\
&+2\sum_{i,\alpha} (B^\alpha_{1i}\omega^i_1-B^\alpha_{2i}\omega^i_2)E_\alpha+\cdots
+2\sum_{i,\alpha}(B^\alpha_{1i}\omega^i_1-B^\alpha_{m_1i}\omega^i_{m_1})E_\alpha\\
=&\sum_{i_1,\alpha_1}\left( (B^{\alpha_1}_{11i_1}-B^{\alpha_1}_{22i_1})+\cdots +(B^{\alpha_1}_{11i_1}-B^{\alpha_1}_{m_1m_1i_1})\right)\omega^{i_1}E_{\alpha_1}\\
&+2\sum_{i_1,\alpha_1}(B^{\alpha_1}_{1i_1}\omega^{i_1}_1 -B^{\alpha_1}_{2i_1}\omega^{i_1}_2)E_{\alpha_1}
+\cdots +2\sum_{i_1,\alpha_1}(B^{\alpha_1}_{11}\omega^{i_1}_1 -B^{\alpha_1}_{m_1i_1}\omega^{i_1}_{m_1})E_{\alpha_1}\in V_{10}
\end{align*}
implying that
\eqref{4.18} holds in this case. Other cases can be similarly but more easily considered. Now, from \eqref{4.18} directly follows Lemma \ref{lem4.3}.
\end{proof}

{\rmk\label{rmk4.1}\rm The conclusion that $\dim V_0$ is constant along $M^m$ can also be directly proved as follows:}

For some fixed $i_1,i_2,i_3$, we have by \eqref{4.17}
\be\label{4-19}\sum_a m_a\lagl B^{V_0}_{i_ai_a},B^{V_0}_{i_bi_b}\ragl_1=0,\quad b=1,2,3\ee
which together with \eqref{4-3} shows that $\lagl B^{V_0}_{i_ai_a},B^{V_0}_{i_bi_b}\ragl_1$ are constant for any $a,b$. It then follows from the Lagrangian identity that, for $a\neq b$,
$$
(B^{V_0}_{i_ai_a}\times B^{V_0}_{i_bi_b})^2 =\lagl B^{V_0}_{i_ai_a},B^{V_0}_{i_ai_a}\ragl_1\lagl B^{V_0}_{i_bi_b},B^{V_0}_{i_bi_b}\ragl_1-\lagl B^{V_0}_{i_ai_a},B^{V_0}_{i_bi_b}\ragl_1^2=\const.
$$
On the other hand, by \eqref{4.17}, $B^{V_0}_{i_1i_1}$ is parallel to $B^{V_0}_{i_2i_2}$ if and only if $B^{V_0}_{i_1i_1}$, $B^{V_0}_{i_2i_2}$, $B^{V_0}_{i_3i_3}$ are parallel to each other, that is, $\dim V_0=1$. The remark is proved.

Hence there are only two cases that need to be considered:

Case 1. $\dim V_0=2$. In this case, we can find two special indices $\alpha_0$ and $\alpha'_0$ such that
\be\label{4.20}
B^{V_0}_{i_1i_1}=B^0_1E_{\alpha_0},\quad B^{V_0}_{i_ai_a}=B^0_aE_{\alpha_0}+B^0_a{}'E_{\alpha'_0},\ \mb{ for any } i_a,\quad a=2,3,
\ee
with
\be\label{4-22} B^0_1B^0_2{}'\neq 0.\ee

{\lem\label{lem4.5} $B^0_1$ and $B^0_2,B^0_2{}',B^0_3,B^0_3{}'$ are constants.}

\begin{proof} First, \eqref{2.14} and \eqref{4.20} give that
\be\label{4-23}
m_1B^0_1+m_2B^0_2+m_3B^0_3=m_2B^0_2{}'+m_3B^0_3{}'=0.
\ee

On the other hand, by \eqref{4-3}, \eqref{4-19} and \eqref{4.20}, we find
\begin{align*}
&m_1(B^0_1)^2=m_2(\lambda_1+\lambda_2)+m_3(\lambda_1+\lambda_3)=\const,\\ &m_2((B^0_2)^2+(B^0_2{}')^2)=m_1(\lambda_1+\lambda_2)+m_3(\lambda_2+\lambda_3)=\const,\\
&m_3((B^0_3)^2+(B^0_3{}')^2)=m_1(\lambda_1+\lambda_3)+m_2(\lambda_2+\lambda_3)=\const.
\end{align*}
This with \eqref{4-3} and \eqref{4-23} easily shows that $B^0_1$ and $B^0_2,B^0_2{}',B^0_3,B^0_3{}'$ are all constants.
\end{proof}

\begin{lem}\label{lem4-2} For the \mo normal frame field $\{E_\alpha\}\equiv\{E_{\alpha_a},E_{\alpha_0},E_{\alpha'_0}\}$ chosen above, the \mo normal connection forms $\omega^\beta_\alpha$ satisfy
\be\label{4-14}
\omega^\beta_{\alpha_0}=\omega^\beta_{\alpha'_0}=\omega^{\beta_b}_{\alpha_a}=0,\quad a\neq b.
\ee
\end{lem}

\begin{proof}
Due to Lemma \ref{lem4.4}, \eqref{4-12} and \eqref{4.20}, we only need to show that $\omega^{\alpha'_0}_{\alpha_0}=0$. In fact, by \eqref{4-5}, we know that $\sum B^{\alpha'_0}_{i_1i_1j}\omega^j=0$ which with Lemma \ref{lem4.5} and \eqref{4.20} shows that
$$
0=dB^{\alpha'_0}_{i_1i_1}-B^{\alpha'_0}_{ji_1}\omega^j_{i_1}-B^{\alpha'_0}_{i_1 j}\omega^j_{i_1}+B^{\alpha_0}_{i_1i_1}\omega^{\alpha'_0}_{\alpha_0} +B^{\alpha_1}_{i_1i_1}\omega^{\alpha'_0}_{\alpha_1}=B^{\alpha_0}_{i_1i_1}\omega^{\alpha'_0}_{\alpha_0},
$$
where we have used $\omega^{\alpha'_0}_{\alpha_a}=0$, $a=1,2,3$, which are directly obtained by Lemma \ref{lem4.4}. Thus $\omega^{\alpha'_0}_{\alpha_0}=0$.
\end{proof}

Let $Y$ and $N$ be the \mo position vector and the \mo biposition vector of $x$, respectively. As done earlier by many authors (for example, \cite{lwz}, \cite{lw1}, \cite{lz05}, \cite{zhw}) and the very recent paper \cite{ls3}, we define another vector-valued function
\be\label{4-15}
{\mathbf c}:=N+\lambda Y+{\mu}_0E_{\alpha_0}+{\mu}'_0E_{\alpha'_0}
\ee
for some constants $\lambda$, ${\mu}_0$ and ${\mu}'_0$ to be determined. Then by using \eqref{2-6} and \eqref{2-8} we find
$$
d{\mathbf c}=\sum_{a,i_a}(\lambda_a+\lambda-{\mu}_0B^0_a-{\mu}'_0B^0_a{}')\omega^{i_a}Y_{i_a}
$$
with $B^0_1{}'=0$. Note that, by \eqref{4-23}
$$
\ldt 1&B^0_1&0\\1&B^0_2&B^0_2{}'\\1&B^0_3&B^0_3{}'\rdt =B^0_2{}'(B^0_1-B^0_3)+B^0_3{}'(B^0_2-B^0_1)
=\fr1{m_3}B^0_1 B^0_2{}'(m_1+m_2+m_3)=\fr m{m_3}B^0_1 B^0_2{}'\neq 0,
$$
the system of linear equations
\be\label{4.25}\left\{
\begin{aligned}&\lambda_1+\lambda-{\mu}_0B^0_1=0,\\
&\lambda_2+\lambda-{\mu}_0B^0_2-{\mu}'_0B^0_2{}'=0,\\
&\lambda_3+\lambda-{\mu}_0B^0_3-{\mu}'_0B^0_3{}'=0,\end{aligned}\right.
\ee
for $\lambda,{\mu}_0,{\mu}'_0$ has a unique solution as
\be\label{4.27}\left\{\begin{aligned}
&\lambda=-\fr{m_1\lambda_1+m_2\lambda_2+m_3\lambda_3}{m},\\
&{\mu}_0=-\fr1{mB^0_1}\left(\sum m_a\lambda_a-m\lambda_1\right),\\
&{\mu}'_0=-\fr1{mB^0_1B^0_2{}'}\left((m_1B^0_1+(m_2+m_3)B^0_2)\lambda_1\right.\\
&\hs{1cm}+\left.((m_1+m_3)B^0_1+m_2B^0_2)\lambda_2+m_3(B^0_1-B^0_2)\lambda_3\right).
\end{aligned}\right.\ee
Thus the following lemma is proved:

\begin{lem}\label{lem4-3} Let $\lambda$, ${\mu}_0$ and ${\mu}_0'$ be given by \eqref{4.27}. Then the vector-valued function ${\mathbf c}$ defined by \eqref{4-15} is constant on $M^m$ and
\be\label{4-17}
\lagl {\mathbf c},{\mathbf c}\ragl=2\lambda+{\mu}_0^2+{\mu}'_0{}^2,\quad \lagl {\mathbf c},Y\ragl=1.
\ee
\end{lem}

Next we have to consider the following three subcases:

Subcase (1): ${\mathbf c}$ is time-like; Subcase (2): ${\mathbf c}$ is light-like; Subcase (3): ${\mathbf c}$ is space-like.

Since the argument that follows here is standard and same as that of \cite{ls3} (see Case 1 there), we omit the detail of it and only state the corresponding conclusions:

{\prop\label{prop4.1} Let $x:M^m\to \bbs^{m+p}$ be as in the main theorem (Theorem \ref{main}). If $\dim V_0=2$ and ${\mathbf c}$ is the constant vector given by \eqref{4-15}, then

$(1)$ ${\mathbf c}$ is time-like and $x$ is \mo equivalent to a non-minimal and umbilic-free pseudo-parallel immersion $\td x:M^m\to \bbs^{m+p}$ with parallel mean curvature and constant scalar curvature, which has three distinct principal curvatures in the direction of the mean curvature vector;

$(2)$ ${\mathbf c}$ is light-like and $x$ is \mo equivalent to the image under $\sigma$ of a non-minimal and umbilic-free pseudo-parallel immersion $\bar x:M^m\to \bbr^{m+p}$ with parallel mean curvature and constant scalar curvature, which has three distinct principal curvatures in the direction of the mean curvature vector;

$(3)$ ${\mathbf c}$ is space-like and $x$ is \mo equivalent to the image under $\tau$ of a non-minimal and umbilic-free pseudo-parallel immersion $\bar x:M^m\to \bbh^{m+p}$ with parallel mean curvature and constant scalar curvature, which has three distinct principal curvatures in the direction of the mean curvature vector.}

Case 2. $\dim V_0=1$.

In this case, there is an index $\alpha_0$ such that $V_0=\bbr E_{\alpha_0}$. Thus we can write
\be\label{4.29}
B^{V_0}_{i_ai_a}=B^0_a E_{\alpha_0},\ \mb{ for each }\ i_a,\quad a=1,2,3.
\ee
It follows that
\be\label{4.30}
m_1B^0_1+m_2B^0_2+m_3B^0_3=0,\quad B^0_aB^0_b=-(\lambda_a+\lambda_b),\ \mb{ for }\ a\neq b\ee
which implies
\be\label{4.31}
(B^0_a)^2=\fr1{m_a}((m_{a'}+m_{a''})\lambda_a+m_{a'}\lambda_{a'}+m_{a''}\lambda_{a''}),\quad a=1,2,3
\ee
where $a,a',a''$ is an even permutation of $1,2,3$.

Furthermore, Lemma \ref{lem4.4} implies in the present case that
\be\label{4.32}
\omega^{\alpha_a}_{\alpha_0}=\omega^{\alpha_b}_{\alpha_a}\equiv 0,\ \mb{ for all $a$ and }b\neq a.
\ee

Define
\be\label{4.33}
z_a=N+\lambda_a Y-B^0_aE_{\alpha_0},\quad a=1,2,3.
\ee
Then, by \eqref{2-6}, \eqref{2-8} and \eqref{4.32}, we find that
\be\label{4.39}
dz_a=\sum_{i,j}A_{ij}\omega^jY_i+\lambda_a\sum_i\omega^iY_i+B^0_a\sum_{i,j}B^{\alpha_0}_{ij}\omega^jY_i =(2\lambda_a+(B^0_a)^2)\sum_{i_a}\omega^{i_a}Y_{i_a},\quad a=1,2,3.
\ee
Thus $z_a$ is constant on $M_b$ for $b\neq a$.

Using \eqref{2-7}, \eqref{2-8}, \eqref{4.32} and \eqref{4.39}, the following lemma is easily proved:

\begin{lem}\label{lem4-4} The subbundles $\bbr z_a$, $Y_*(TM_a)$, $\bbr E_{\alpha_0}$, $V_{a0}$, $a=1,2,3$, are mutually orthogonal, and the \mo normal connection on the \mo normal bundle $V$ is the direct sum of its restrictions on $\bbr E_{\alpha_0}$, $V_{a0}$, $a=1,2,3$. Moreover,
$$\bbr z_a\oplus Y_*(TM_a)\oplus V_{a0},\quad a=1,2,3$$
are orthogonal to each other in $\bbr^{m+p+2}_1$ and are constant on $M_a$, $a=1,2,3$, respectively.
\end{lem}

Subcase (i): One of $2\lambda_1+(B^0_1)^2$, $2\lambda_2+(B^0_2)^2$ and $2\lambda_3+(B^0_3)^2$ vanishes.

Without loss of generality, we assume $2\lambda_1+(B^0_1)^2=0$. Then by \eqref{4.39}, $dz_1\equiv 0$ and thus $z_1={\mathbf c}$ is a constant vector on $M^m$. Furthermore,
$$
\lagl {\mathbf c},{\mathbf c}\ragl_1=2\lambda_1+(B^0_1)^2=0,\quad \lagl {\mathbf c},Y\ragl_1=1.
$$
Therefore the according argument in \cite{ls3} (see Subcase (ii) of Case 1 there) applies to the present case and proves the following conclusion:

{\prop\label{prop4.2} Let $x:M^m\to \bbs^{m+p}$ be as in the main theorem (Theorem \ref{main}). If $\dim V_0=1$ and there exists some $a$, $1\leq a\leq 3$, such that $2\lambda_a+(B^0_a)^2=0$, then $x$ is \mo equivalent to

$(2)$ the image under $\sigma$ of a non-minimal and umbilic-free pseudo-parallel immersion $\bar x:M\to \bbr^{m+p}$ with parallel mean curvature and constant scalar curvature, which has three distinct principal curvatures in the direction of the mean curvature vector.}

Subcase (ii): $2\lambda_a+(B^0_a)^2\neq 0$, $a=1,2,3$.

In this subcase, by using \eqref{4.30} and \eqref{4.31} the following lemma can be easily proved by a direct computation:

{\lem\label{lem} The three constants $B^0_a$, $a=1,2,3$, have the properties that
\be\label{4.35}
\sum_a\fr1{2\lambda_a+(B^0_a)^2}=\sum_a\fr{B^0_a}{2\lambda_a+(B^0_a)^2}=0,\quad \sum_a\fr{\lambda_a}{2\lambda_a+(B^0_a)^2}=1.
\ee}

\begin{rmk}\rm Note that $\bbr E_{\alpha_0}$, $\bigoplus_a (Y_*(TM_a)\oplus V_{a0})$ are space-like, and
\be\label{4.36}
\lagl z_a,z_a\ragl_1=2\lambda_a+(B^0_a)^2\neq 0,\quad a=1,2,3.
\ee
It follows that there exists one and only one index $a$ such that
$$\lagl z_a,z_a\ragl_1<0,\ \mb{ or equivalently, }\ 2\lambda_a+(B^0_a)^2<0.$$
With no loss of generality we assume that
\be\label{4.37}
r^2_1:=-\fr1{2\lambda_1+(B^0_1)^2},\quad r^2_a:=\fr1{2\lambda_a+(B^0_a)^2},\quad a=2,3
\ee
for positive numbers $r_1,r_2,r_3$. Then by \eqref{4.30}, \eqref{4.31} and \eqref{4.35} we have
\be\label{4.38}
r^2_1=r^2_2+r^2_3,\quad m_3r^2_2\neq m_2r^2_3.
\ee
\end{rmk}

Now from the \mo second fundamental form $B$, we define for each $a$
$$
\Baa=\sum B^{\alpha_a}_{i_aj_a}\omega^{i_a}\omega^{j_a}E_{\alpha_a}.
$$
Then $\Baa$ is a $V_{a0}$-valued symmetric $2$-form on $M_a$ with components $\Baa^{\alpha_a}_{i_aj_a}=B^{\alpha_a}_{i_aj_a}$.

Let $\Baa^{\alpha_a}_{i_aj_a,k_a}$ be the components of the covariant derivatives of $\Baa$ with the induced connection on $V_{a0}$. Then, as the consequence of \eqref{4-2}, \eqref{4.32} and
Lemma \ref{lem4-4}, we have
\be\label{4-36}
\Baa^{\alpha_a}_{i_aj_a,k_a}=B^{\alpha_a}_{i_aj_ak_a}.
\ee

Since $B^{\alpha_b}_{i_aj_a}=0$ for $b\neq a$, the vanishing of the \mo form $C$ together with \eqref{2.15}, \eqref{2-16}, \eqref{2.18}, \eqref{4-1}, \eqref{4-2} and \eqref{4-36} proves the following lemma:

\begin{lem}\label{lem4-5}
The Riemannian manifold $(M_a,g^{(a)})$ and the vector bundle valued symmetric tensor $\Baa$ satisfies the Gauss equation, Codazzi equation and Ricci equation for submanifolds in a space form of constant curvature $2\lambda_a+(B^0_a)^2$. Namely
\begin{align}
&R_{i_aj_ak_al_a}=\sum (\Baa^{\alpha_a}_{i_al_a}\Baa^{\alpha_a}_{j_ak_a} -\Baa^{\alpha_a}_{i_ak_a}\Baa^{\alpha_a}_{j_al_a}) +(2\lambda_a+(B^0_a)^2)(\delta_{i_al_a}\delta_{j_ak_a}-\delta_{i_ak_a}\delta_{j_al_a}),\label{4-37}\\
&\Baa^{\alpha_a}_{i_aj_a,k_a}=\Baa^{\alpha_a}_{i_ak_a,j_a},\quad
R^\bot_{\alpha_a\beta_ai_aj_a}=\sum (\Baa^{\alpha_a}_{j_ak_a}\Baa^{\beta_a}_{i_ak_a}-\Baa^{\alpha_a}_{i_ak_a}\Baa^{\beta_a}_{j_ak_a}).
\end{align}
\end{lem}

By Lemma \ref{lem4-5}, there exist an isometric immersion
$$\td y\equiv(\td y_0,\td y_1):(M_1,g^{(1)})\to\bbh^{m_1+p_1}\left(-\fr1{r^2_1}\right)\subset\bbr^{m_1+p_1+1}_1$$
with $\Ba$ as its second fundamental form, and two isometric immersions
$$\td y_a:(M_a,g^{(a)})\to\bbs^{m_a+p_a}(r_a)\subset\bbr^{m_a+p_a+1},\quad a=2,3$$
with $\Baa$ as their second fundamental forms, respectively.

Note that $B^{\alpha_b}_{i_aj_a}\equiv 0$ for $b\neq a$. It follows from \eqref{2.14} that both $\td y$ and $\td y_a$, $a=2,3$, are minimal immersions. Furthermore, if denote by $\td S_a$ the scalar curvatures of $M_a$, then by \eqref{4.37}, \eqref{4-37} and the minimality, we have
\be\label{4-39}
 \td S_1=-\fr{m_1(m_1-1)}{r^2_1}-\sum (B^{\alpha_1}_{i_1j_1})^2,\quad\td S_a=\fr{m_a(m_a-1)}{r^2_a}-\sum (B^{\alpha_a}_{i_aj_a})^2,\quad a=2,3
\ee
showing that
\begin{align}
\td S_1+\fr{m_1(m_1-1)}{r^2_1}=&-\sum(B^{\alpha_1}_{i_1j_1})^2 \leq 0,\label{4-39.1}\\
\td S_a-\fr{m_a(m_a-1)}{r^2_a}=&-\sum(B^{\alpha_a}_{i_aj_a})^2 \leq 0,\quad a=2,3.\label{4-39.2}
\end{align}

On the other hand, by \eqref{2.14},
\be\label{4.46}
\sum_{a,i_a,j_a}(B^{\alpha_a}_{i_aj_a})^2=\sum_{\alpha,i,j}(B^\alpha_{ij})^2 -\sum_am_a(B^0_a)^2=\fr {m-1}m-\sum_am_a(B^0_a)^2=\const.
\ee
Thus by \eqref{4-39} and \eqref{4.46},
\begin{align}
\td S_1+\td S_2+\td S_3=&-\fr{m_1(m_1-1)}{r_1^2}+\fr{m_2(m_2-1)}{r_2^2}+\fr{m_3(m_3-1)}{r_3^2}\nnm\\
&-\fr{m-1}{m}+m_1(B^0_1)^2 +m_2(B^0_2)^2+m_3(B^0_3)^2\nnm\\
=&\const.
\label{S1+S2+S3}
\end{align}

Since $\td S_a$'s are functions defined on $M_a$'s, respectively, it follows that all $\td S_a$'s are constant on $M^m$ and, by \eqref{4-39.1}, \eqref{4-39.2}, we can write
\begin{align}
\td S_1=&-\fr{m_1(m_1-1)}{r^2_1}+{\mu}_1\left(m_1(B^0_1)^2
+m_2(B^0_2)^2+m_3(B^0_3)^2-\fr{m-1}{m}\right)\label{4.47}\\
\td S_a=&\fr{m_a(m_a-1)}{r^2_a}+{\mu}_a\left(m_a(B^0_1)^2
+m_2(B^0_2)^2+m_3(B^0_3)^2-\fr{m-1}{m}\right),\quad a=2,3\label{4.48}
\end{align}
for some positive constants ${\mu}_1,{\mu}_2,{\mu}_3$ satisfying ${\mu}_1+{\mu}_2+{\mu}_3=1$.

Now let ${\rm LS}(\fkm,\fkp,\fkr,{\mu})$ be one of the submanifolds in Example \ref{expl3.2} defined by $\td y$, $\td y_2$ and $\td y_3$. Then it is not hard to see that ${\rm LS}(\fkm,\fkp,\fkr,{\mu})$ has the same \mo metric $g$ and the same \mo second fundamental form $B$ as those of $x$. Furthermore, by choosing the normal frame field $\{\td e_\alpha\}$ as given in \eqref{3-10}--\eqref{3-11.2} where, in the present case,
$$\bar E_\alpha=E_\alpha,\quad m+1\leq \alpha\leq m+p,$$
we compute directly:
$$
\td\omega^\beta_\alpha=d\td e_\alpha\cdot\td e_\beta=\lagl dE_\alpha,E_\beta\ragl_1=\begin{cases}\omega^\beta_\alpha,& \mb{for either } m+1\leq\alpha,\beta\leq m+p_1,\\
&\mb{or } m+p_1+1\leq\alpha,\beta\leq m+p_1+p_2,\\
&\mb{or } m+p_1+p_2+1\leq\alpha,\beta\leq m+p;\\\\
0,& otherwise,
\end{cases}
$$
implying that $x$ and ${\rm LS}(\fkm,\fkp,\fkr,{\mu})$ have the same \mo normal connection. Therefore, by Theorem \ref{wth}, $x$ is \mo equivalent to ${\rm LS}(\fkm,\fkp,\fkr,{\mu})$. So we have proved the following proposition:

{\prop\label{prop4.3} Let $x:M^m\to \bbs^{m+p}$ be as in the main theorem (Theorem \ref{main}). If $\dim V_0=1$ and $2\lambda_a+(B^0_a)^2\neq 0$, $a=1,2,3$, then $x$ is \mo equivalent to

$(4)$ a submanifold ${\rm LS}(\fkm,\fkp,\fkr,{\mu})$ given in Example \ref{expl3.2} for some multiple parameters $\fkm,\fkp,\fkr,{\mu}$ satisfying $m_3r^2_2\neq m_2r^2_3$.}

\vs{.5cm}
{\em The proof of the main theorem (Theorem \ref{main})}.

As discussed earlier in this section, there are only the following two cases with additional subcases that need to be considered:

(1) $\dim V_0=2$.

(2) $\dim V_0=1$:

Subcase (i), one of $2\lambda_a+(B^0_a)^2$ ($a=1,2,3$) vanishes;

Subcase (ii), $2\lambda_a+(B^0_a)^2\neq 0$, $a=1,2,3$.

Thus the main theorem follows directly from Propositions \ref{prop4.1}, \ref{prop4.2} and \ref{prop4.3}.

\end{document}